\documentclass[11pt]{amsart}

\usepackage[english]{babel}
\usepackage{amsmath,amssymb,amsthm}

\newcommand{\R}{\mathbb R}
\newcommand{\Half}{\mathbb H} 
\newcommand{\Disk}{\mathbb{D}} 
\newcommand{\eps}{\epsilon} 
\renewcommand{\Pr}{\mathbb{P}}
\newcommand{\Prob}[1]{\Pr\left\{#1\right\}}
\newcommand{\E}{\mathbb{E}} 
\newcommand{\bd}{\partial} 
\newcommand{\di}{\, d}
\newcommand{\dd}{d}
\newcommand{\dist}  {{\rm dist}} 
\newcommand{\cGF}{\overline{G}} 
\newcommand{\SK}{\mathcal{H}} 
\renewcommand{\Im}{{\rm Im}} 
\renewcommand{\Re}{{\rm Re}} 
\newcommand{\crad}{\Upsilon} 
\newcommand {\cradt}{\tilde \Upsilon} 
\renewcommand{\phi}{\varphi}

\newcommand{\eref}{\eqref}

\newtheorem{theorem}{Theorem}[section]
\newtheorem{proposition}[theorem]{Proposition}
\newtheorem{lemma}[theorem]{Lemma}

\theoremstyle{remark}

\begin{document}

\title[The Green's function for radial SLE]{The Green's function for the radial
Schramm-Loewner evolution}

\author{Tom Alberts}
\address{Department of Mathematics, California Institute of Technology, Pasadena, CA 91125, USA}
\email{alberts@caltech.edu}

\author{Michael J Kozdron}
\address{Department of Mathematics \& Statistics, University of Regina, Regina, SK S4S 0A2, Canada}
\email{kozdron@stat.math.uregina.ca}

\author{Gregory F Lawler}
\address{Department of Mathematics, University of Chicago, Chicago, IL 60637, USA}
\email{lawler@math.uchicago.edu}

\begin{abstract}
We prove the existence of the Green's function for radial  SLE$_\kappa$ for $\kappa < 8$. Unlike the chordal case where an explicit formula for the Green's function is known for all values of $\kappa<8$, we give an explicit formula only for $\kappa = 4$. For other values of $\kappa$, we give a formula in terms of an expectation with respect to SLE conditioned to go through a point.
\end{abstract}

\subjclass[2010]{60J67,  82B31}

\maketitle

\section{Introduction}

The Schramm-Loewner evolution is a one-parameter family of two-dimensional  random growth processes  that was introduced in 1999 by the late Oded Schramm~\cite{Sch00}.
Denoting the parameter by $\kappa > 0$, these random growth processes (which are abbreviated as SLE$_\kappa$) have provided a valuable mathematical tool for obtaining rigorous results about a variety of discrete lattice models from statistical mechanics. Assuming appropriate boundary conditions, these discrete models contain a random curve which converges to SLE$_\kappa$  for some value of $\kappa$ (where $\kappa$ depends on the model being considered). For instance, the  exploration path
for critical site percolation on the triangular lattice converges to SLE with $\kappa=6$, contour lines in the discrete Gaussian free field converge to SLE with $\kappa=4$, a perimeter curve for the uniform spanning tree converges to SLE with $\kappa=8$, and cluster interfaces in the spin Ising and FK-Ising models at criticality converge to SLE with parameters $\kappa=3$ and $\kappa=16/3$, respectively. Moreover, loop-erased random walk converges to SLE$_2$ and there is strong evidence to suggest that self-avoiding walk converges to SLE$_{8/3}$. For more details
see \cite{SLEbook,ParkCity} and references therein.

In addition to studying SLE as the scaling limit of interfaces in discrete statistical mechanics models, it is of intrinsic interest to understand the path properties of SLE; that is, properties of SLE as a continuous random process in the complex plane.  For example,  one might ask about the Hausdorff dimension of the SLE trace or the dimension of various  random subsets of the trace, probabilities for different events such as the intersection of the trace with either a random or deterministic subset of the complex plane, or the distributions of certain functionals of the trace.

This paper studies the Green's function for radial SLE$_\kappa$, $0<\kappa<8$, from $1$ to $0$ in the unit disk $\Disk$ which is the  (normalized) probability that a radial SLE trace passes near a given point $z \in \Disk$.   The Green's function for chordal SLE$_\kappa$ from $0$ to $\infty$ in the upper half plane $\Half$, denoted by   $\cGF_{\Half}(z;0,\infty)$, is well understood.
It can be defined up to a muliplicative
constant by  the limit
\begin{equation}\label{chordalGFdefn}
\lim_{\eps\to0} \eps^{d-2}\Prob{\Upsilon_\infty(z) \le \eps} = c^* \cGF_{\Half}(z;0,\infty)
\end{equation}
where $d=1+\kappa/8$ is the Hausdorff dimension of the SLE trace $\gamma(0,\infty)$ and $c^*$ is a 
constant.
Here $\Upsilon_\infty(z)$ denotes
one-half times the conformal radius
of (the connected component containing $z$ of)
$\Half \setminus \gamma_\infty
$ with respect to $z$ and $\gamma_\infty$ denotes the trace $\gamma[0,\infty)$.
It follows from the Koebe one-quarter theorem
and Schwarz lemma that
\[  \frac 12 \, \Upsilon_\infty(z) \leq
\dist(z,\gamma_\infty \cup \R) \leq 2 \, \Upsilon_\infty
(z).\]  Furthermore, the exact values
of $c^*$ and $\cGF$ are known,
 see Section \ref{SectChordalGFReview}.
 (Of course, \eref{chordalGFdefn} only
 defines $c^*, \cGF_\Half$ up to a multiplicative
 constant, but there is a natural choice of
 the constant that makes $\cGF_\Half$ simplest.)
 In this paper, we consider the analogue for
 radial SLE$_\kappa$.

The outline of the remainder of the paper is as follows.  In Section~2.1 we review the basic facts about the Green's function for chordal SLE while in Section 2.2 we introduce the Green's function for radial SLE.  In Section~3 we derive the partial differential equation for the Green's function for radial SLE in $\Disk$ assuming that the Green's function exists. This PDE can be solved explicitly when $\kappa=4$ and so we state in Theorem~3.2 a formula for the  Green's function for radial SLE$_4$ from 1 to 0 in $\Disk$, namely
\[  G_{\Disk}(z;1) = \sqrt{\frac{1-|z|^2}{|z|\cdot|1-z|^2}}\]
for $z \in \Disk$.
 In order to study the radial Green's function for other values of $\kappa$ we work in the
 half-infinite cylinder $\Half^*$ obtained by taking the upper half plane $\Half$ and identifying points $w_1$, $w_2$ such that $w_1 -w_2 = k\pi$ for some integer $k$. The details are given in Section~4, and we state the
 main theorem,   Theorem~\ref{main_thm},
 which gives a representation of the radial SLE$_\kappa$ Green's function in $\Half^*$ for all $\kappa<8$ in terms of an expectation with respect to SLE conditioned to go though a point. Finally, in Section~5 we prove the theorem by
 showing that the corresponding limit in~\eref{chordalGFdefn} exists in the radial
 case.

\section{The Green's function for SLE}

\subsection{Chordal SLE Green's function}\label{SectChordalGFReview}

Chordal SLE$_\kappa$ for $\kappa>0$ is a probability
measure on curves $\gamma$
connecting distinct boundary points of a domain $D$.  It is characterized
by conformal invariance and the domain Markov property.  In the upper
half plane it can be obtained by considering the Loewner equation
\begin{equation} \label{chordalL}
   \partial_t g_t(z) = \frac{a}{g_t(z) - B_t} , \;\;\;\; g_0(z) = z,
   \end{equation}
where $a=2/\kappa$ and $B_t$ is a standard one-dimensional Brownian motion.  This equation
is valid for all $t < T_z$ and the curve $\gamma$ can be described
by saying that $\{z: T_z > t\}$ is the unbounded component
of $\Half \setminus \gamma_t$ where $\gamma_t = \gamma[0,t]$.
It also satisfies $g_t(\gamma(t)) = B_t$.  It is known that for $0<\kappa \le 4$ the curves are simple, while for $d \geq 8$, the
curves are plane filling.  In this paper, we restrict to $\kappa < 8$.

Suppose  $d$ is the fractal dimension of the curve $\gamma$.  Then we would expect
that there is a function $\cGF(z)$ such that
\[
     \Prob{\dist(z,\gamma_\infty) \leq \epsilon} \sim \cGF(z) \, \epsilon^{2-d} \;\;\;
{\rm as} \;\;\;  \epsilon \downarrow 0.
\]
 Rohde and Schramm~\cite{rohde_schramm} first noticed that if such a $\cGF$ existed
then
\[\overline{M}_t(z) = |g_t'(z)|^{2-d} \, \cGF(g_t(z) - B_t)\]
 would have to be a local martingale.
Using  the scaling rule $\cGF(rz) = r^{d-2} \, \cGF(z)$ along with It\^o's formula, they
showed that this implies (up to a multiplicative constant) that
\[             \cGF(z) = [\Im(z)]^{d-2} \, \sin^{\frac 8 \kappa - 1}(\arg z)\]
where $d = 1 +  \kappa/ 8$.
Using this idea Beffara~\cite{Bef} (see also~\cite{LawWerGF}) proved that the
Hausdroff dimension of the paths is $d$.  It is still not known whether or not
$ \Prob{\dist(z,\gamma_\infty) \leq \epsilon } \sim c\,\cGF(z) \, \epsilon^{2-d}$ for
some $c$, but a similar result is known which we now describe. 
Let $\Upsilon_D(z)$ denote one-half the conformal radius of $z$ in 
$D$.  For simply
connected domains $D$,  which is all we
will need,  this means $\Upsilon_\Half(z)  =
\Im(z)$ and $\Upsilon$ transforms under conformal transformations by
$\Upsilon_{f(D)}(f(z)) = |f'(z)| \, \Upsilon_D(z)$.
 Using the Schwarz lemma and
the Koebe one-quarter theorem, one can see that
\[              \frac 12 \, \dist(z,\partial D) \leq \Upsilon_D(z) \leq 2 \
  \dist(z,\partial D).
\]
If $D$ is not connected, we write $\Upsilon_D(z)$
for the conformal radius of $z$ in the connected
component of $D$ containing $z$.

The Green's function for chordal SLE$_\kappa$ in a simply connected domain $D$ connecting
boundary points $w_1$, $w_2$ is defined by
\[    \cGF_D(z;w_1,w_2) = \Upsilon_D(z)^{d-2} \, S_D(z;w_1,w_2)^{\frac 8 \kappa - 1}. \]
Here $S_D(z;w_1,w_2) = \sin [\arg F(z)]$ where $F:D \rightarrow \Half$ is a conformal
transformation with $F(w_1) = 0$, $F(w_2) = \infty$.
 Note that $\cGF(z) = \cGF_{\Half}(z;0,\infty)$.
Moreover, if $\gamma$ is an SLE curve from $w_1$ to $w_2$ in $D$,
\begin{equation}  \label{may15.1}
\lim_{\epsilon \downarrow 0} \epsilon^{d-2} \, \Prob{\Upsilon_{D \setminus \gamma_\infty}(z)
 \leq \epsilon } = c^* \, \cGF_D(z;w_1,w_2),
  \end{equation}
where
\begin{equation}  \label{defncstar}
 c^* =  2 \, \left[\int_0^\pi \sin^{ 8 /\kappa} \, x
 \di x \right]^{-1} .
 \end{equation}
A proof of this can be found
in  Lemma~2.10 of~\cite{LawWerGF}; we will review the proof in Section \ref{SectMainThm}.

The chordal Green's function satisfies the scaling rule
\[                \cGF_D(z;w_1,w_2) = |f'(z)|^{2-d} \, \cGF_{f(D)}
    (f(z);f(w_1),f(w_2)) . \]
Moreover,
\[   \overline{M}^D_t(z) := \cGF_{D \setminus \gamma_t}(z;\gamma(t),w_2) \]
is a local martingale.

\subsection{Radial SLE Green's function}

Recall that radial SLE$_\kappa$ for $\kappa>0$ is a probability measure on curves connecting a boundary point and
an interior point.  For ease, and without loss of generality, we will choose the
interior point to be the origin.  If $D$ is a simply connected domain containing
the origin and $w \in \partial D$,
we define $G_D(z;w)$, the {\em Green's function for radial SLE$_\kappa$}
for $\kappa < 8$, by
\begin{equation}  \label{may15.2}
  \lim_{\epsilon \downarrow 0} \epsilon^{d-2}\,  \Prob{\Upsilon_{D \setminus \gamma_\infty}(z) \leq \epsilon} = c^*G_D(z;w)
\end{equation}
where $\gamma$ is a radial SLE$_\kappa$ path from $w$ to $0$ and $c^*$ is
the same constant as in~\eref{defncstar}.  In 
Theorem~\ref{main_thm}  we establish
the existence of a function
satisfying~\eref{may15.2}.  It suffices to
prove existence for $D = \Disk$, $w = 1$.  Indeed
if~\eref{may15.2} holds in this case, and
$f: D \rightarrow \Disk$ is a conformal transformation
with $f(0) = 0$, $f(w) = 1$, then conformal
invariance shows that~\eref{may15.1} holds with
\[         G_D(z;w) = |f'(0)|^{2-d} \, G_\Disk(f(z);1).\]
As in the chordal case, we note that if such a function exists, then
\[     M_t^D(z) := G_{D \setminus \gamma_t}(z;\gamma(t)) \]
is a local martingale. Indeed, we could have
used this property as the definition (up to a multiplicative constant) of the Green's function.
  If $G$ is sufficiently smooth, see Sections~\ref{SectRadialGFinDisk.nowsect} and~\ref{SectRadSLEinHstar}, 
It\^o's formula  gives
a partial differential equation for $G$.  
In Section~\ref{SectMainThm}
we show that~\eref{may15.2} holds.  We summarize
the results here.

  We
write $\Half^*$ for the half-infinite cylinder obtained by taking the upper half
plane $\Half$ and identifying points $w_1$, $w_2$ such that $w_1 - w_2 = k \pi$ for
some integer $k$. Let $\psi(z) = e^{2iz}$
which is a conformal transformation of
 $\Half^*$ onto $\Disk \setminus \{0\}$.  Let
 $\phi(z) = (z-i)/(z+i)$ which is a conformal transformation of $\Half$ onto $\Disk$. Let
\begin{equation}  \label{jun5.1}
  f(z) = \psi^{-1} \circ \phi (z) =
        \frac{1}{2i} \, \log\left[\frac{z-i}{z+i}
        \right]
        \end{equation}
and set $\Disk^+_r = \{z \in \Half: |z| < r\}$ so that
$f$ is a conformal transformation of $\Disk^+_1$
onto a domain $f(\Disk^+_1)$ with  $f(0) = 0$, $f'(0) = 1$.
We write $G(z) = G_{\Half^*}(z;0)$ for the radial SLE Green's function in $\Half^*$ which satisfies
\begin{equation}\label{GFinHstar}
G(z) = G (x+iy) =  2^{2-d} \, e^{-2(2-d)y} \, G_{\Disk}(e^{-2y + i2x};1).
\end{equation}
In analogy with the chordal case, we will write
\[            G_{D}(z;w) = \Upsilon_D(z)^{d-2} \, F_D(z;w) , \]
where $F_D(z;w)$ is a conformal {\em invariant}.  For chordal SLE$_\kappa$,
invariance under $z \mapsto rz$ showed that $ r^{2-d} \, \cGF(rz)$ is
a function of $\arg z$, and hence the solution could be found from a
one-variable differential equation.  It is a little more complicated
in the radial case.  We will write
\[       G(x+iy) = \Upsilon_{\Half^*}(x+iy)^{d-2} \, F(z) =
     [\sinh y \, \cosh y]^{d-2} \, F(z) .
\]
The second equality can be obtained by conformal transformation, using
$\Upsilon_\Disk(z) = (1-|z|^2)/2$.  There is no scaling relation
that allows $F$ to be written as a function of one real variable.  At
the moment, except in the case $\kappa = 4$, we do not have an
explicit expression for $G$.  However, we can write
\begin{equation}\label{FormulaForGFinHstar}
        G(z) = G(x+iy) =  [\sinh y \, \cosh y]^{d-2}\, \Lambda(z)^{\frac 8 \kappa - 1}  \, \Phi(z),
\end{equation}
where $\Lambda(z)$ is an analogue of $S(z) = \sin(\arg z)$, namely
\begin{equation}\label{defnlambda}   \Lambda(z) = \Lambda (x+iy) = \frac{\sinh y \, \cosh y}{|\sin z|}, \end{equation}
and $\Phi(z)$ is a correction term.  
We can write $\Phi$ as 
\begin{equation}\label{defnq}
\Phi(z) = \E^*\left[g_T'(0)^{q}
\right], \; {\rm where} \;\; q = \frac{(4-\kappa)(\kappa - 8)}
{8 \kappa}.
\end{equation}
We quickly explain the notation above;
see Section~\ref{SectRadSLEinHstar} for more
detail.
  Suppose
$\gamma$ is a radial SLE$_\kappa$ path from $1$ to the origin.  Let $\E^* = \E^*_w$
denote  expectation with
respect to the probability
measure  obtained by conditioning
the path to go through the point $w = e^{2iz}$. See~\cite{LawlerMultiFractAnalysis} for an explanation of how this is made precise using the Girsanov theorem to tilt (or weight) by a particular local martingale.  Let
$T$ be the time at which this path reaches $w$,
and then  $g_T$ is the conformal transformation of the
component of $\Disk \setminus \gamma_T$ containing
the origin to $\Disk$ with $g_T(0) = 0$, $g_T'(0)> 0$.

Note that $\Phi \equiv 1$ if $\kappa = 4$.   If $z = x+iy
\in \Half$ with $|x| \leq \pi/2$ and $|y| \leq 1$, then $G(z) \asymp \cGF(z)$,
where on the left we write $z$ for the corresponding point
on the cylinder $\Half^*$.  Moreover,
$G(z) \sim \cGF(z)$ as $z \rightarrow 0$.

\section{The PDE for the Green's function for radial SLE in $\Disk$}\label{SectRadialGFinDisk.nowsect}

For now and the rest of the paper we fix
 $\kappa < 8$ and write
\[             a = \frac 2 \kappa > \frac 14 .\]

Suppose that $\gamma:[0,\infty) \to \Disk \setminus \{0\}$ is a radial SLE$_\kappa$ path with $\gamma(0)=1$ and $\gamma(t) \to 0$ as $t \to \infty$. To be specific, let $D_t$ denote the 
connected component of $\Disk \setminus
\gamma_t$ containing the origin,
and let $g_t: D_t \to \Disk$ be the conformal transformation with $g_t(0)=0$ and $g_t'(0)=e^{2at}$. Then $g_t(z)$ satisfies the differential equation
\[ \bd_t g_t(z) = 2a g_t(z) \frac{e^{i2B_t} + g_t(z)}{e^{i2B_t} - g_t(z)}, \;\;\; g_0(z) = z, \]
where $B_t$ is a standard one-dimensional Brownian motion.
Note that $g_t(\gamma(t)) = e^{i 2 B_t}$
and $g_t'(0) = e^{2at}$. Moreover, if $\tilde g_t(z) = e^{-i2B_t} g_t(z)$, then
\begin{equation}\label{RadialLE}
\bd_t \log g_t(z) = 4a\pi \SK_{\Disk}(\tilde g_t(z)),
\end{equation}
where $\SK_{\Disk}$ is the Schwarz kernel for $\Disk$  given by
\begin{equation}\label{SKforDisk}
\SK_{\Disk}(z)=
\SK_{\Disk}(z,1) = \frac{1}{2\pi}\frac{1+z}{1-z} = \frac{1}{2\pi} \frac{1-|z|^2}{|1-z|^2}  + \frac{i}{2\pi} \frac{z-\overline{z}}{|1-z|^2}.
\end{equation}
Note that $u_{\Disk}(z) = \Re [\SK_{\Disk}(z)]$ is the Poisson kernel for $\Disk$ and $v_{\Disk}(z) = \Im [\SK_{\Disk}(z)]$ is its harmonic conjugate.  Here, and in what follows, we write the derivative of the logarithm as a convenient shorthand, namely
\begin{equation}\label{logdefn}
\bd_t \log g_t(z) = \frac{\bd_t  g_t(z) }{g_t(z)}.
\end{equation}
Since locally we can take a logarithm of $g_t$, there is no difficulty taking a branch cut so that the left side of~\eref{logdefn} is well-defined. Consequently,
the stochastic differential equation for $\tilde g_t(z)$ is
\begin{equation}\label{RadialLESDE}
\dd \log \tilde g_t(z) = 4a\pi \SK_{\Disk}(\tilde g_t(z)) \di t- 2i \di B_t.
\end{equation}
Moreover, taking spatial derivatives (using $'$ to denote $\bd_z$) of~\eref{RadialLE}  implies that
\[ \bd_t \log g_t'(z) = 4a\pi \left[  \SK_{\Disk}(\tilde g_t(z)) + \tilde g_t(z) \SK_{\Disk}'(\tilde g_t(z)) \right],\]
\begin{equation}\label{DEforgtilde}
\bd_t \log |\tilde g_t'(z)| = 4a\pi\,
\left[ u_{\Disk}(\tilde g_t(z)) +   \Re \left[  \tilde g_t(z) \SK_{\Disk}'(\tilde g_t(z)) \right]\right] ,
\end{equation}
\begin{equation}\label{DEno2}
 \frac{\partial_t |\tilde g_t'(z)|^{2-d}}
 {(4a-1)\pi |\tilde g'_t(z)|^{2-d}}
  =  u_{\Disk}(\tilde g_t(z)) +   \Re \left[  \tilde g_t(z) \SK_{\Disk}'(\tilde g_t(z)) \right] .
\end{equation}
We write $G_\Disk(z) = G_\Disk(z;1)$.
We will find the PDE for $G_\Disk$
such that
  $M_t^{\Disk} := |\tilde g_t'(z)|^{2-d} G_{\Disk}(\tilde g_t(z))$ is a local martingale.
  Suppose that we write $z=re^{i\theta} \in \Disk$ and
\[\log \tilde g_t(z) = U_t + i V_t.\]
As a consequence of~\eref{RadialLESDE}, we find
\[\partial_t U_t = 4a \pi u_{\Disk}(\tilde g_t(z)),  \;\;\; \;
\dd V_t = 4a \pi v_{\Disk}(\tilde g_t(z)) \di t  - 2 \di B_t. \]
 Define $H$ by
$H(r,\theta) = r^{2-d}\, G_{\Disk}(re^{i\theta})$. Assuming sufficient smoothness on $H$ and
using It\^o's formula, we see that $M_t^\Disk$
is a local martingale if and only if $ H$ satisfies
 \begin{equation}\label{thePDEinDisk3}
  H_{\theta\theta} + aFH_{\theta} + aJ H_{r}  + \left(a-\frac{1}{4}\right) F_\theta H =0,
\end{equation}
where
\[F(r,\theta) =  \frac{2r \sin \theta}{1+r^2-2r\cos \theta}, \;\;\;\;
 J(r,\theta) = \frac{r(1-r^2)}{1+r^2-2r\cos \theta}.\]
One might hope to
 find a  solution of~\eref{thePDEinDisk3}  the form
\begin{equation}\label{Hguessp}
H= H(r,\theta) = u_{\Disk}(r,\theta)^p = \left( \frac{1-r^2}{1+r^2-2r\cos \theta} \right)^{p}
\end{equation}
for some $p$.  If $H$ is of form \eref{Hguessp}
then
\begin{equation}\label{thePDEinDisk4}
 H_{\theta\theta} + aFH_{\theta} + aJ H_{r}  + \left(a-\frac{1}{4}\right) F_\theta H 
=(a-p)F H_\theta +
\left(ap-p+a-\frac{1}{4}\right)F_\theta H.
\end{equation}
The right side of~\eref{thePDEinDisk4} equals 0 only if
$ap-p+a-\frac{1}{4}=0$ and $a-p=0$. Since the only solution of this pair of equations is $p=a=1/2$, we conclude that $H(r,\theta)$ given by~\eref{Hguessp} satisfies~\eref{thePDEinDisk3} if and only if $p=a=1/2$. Since $a=1/2$ corresponds to $\kappa=4$ (in which case $2-d=1/2$) we have now found an explicit formula for the Green's function (up to multiplicative
constant) for radial SLE$_4$ from 1 to 0 in $\Disk$:
\[  G_{\Disk}(z) = \sqrt{\frac{1-|z|^2}{|z|\cdot|1-z|^2}}=  r^{-1/2} \left( \frac{1-r^2}{1+r^2-2r\cos \theta} \right)^{1/2},\]
where $z = re^{i\theta}.$

It is certainly true that $H(r,\theta)=u_{\Disk}(r,\theta)^p$ is not the only natural guess for the solution of~\eref{thePDEinDisk3}.  One other guess might be
$u_{\Disk}(r,\theta)^p v_{\Disk}((r,\theta))^q$ for some values of $p$, $q$. It turns out that this guess does not lead to anything useful for analysis. In fact, we will see in Section~\ref{SectRadSLEinHstar} that $\Half^*$ is the best coordinate system for analyzing the Green's function for radial SLE.

\section{Radial SLE$_\kappa$ in the cylinder $\Half^*$}\label{SectRadSLEinHstar}

In the previous section, we saw that it was easier to write the equation in polar coordinates.  We continue this here; we will write points in $\Disk$ as $e^{2iz}$ where $z \in \Half^*$.
Radial SLE$_\kappa$ in $\Half^*$ can be given as the solution to the
Loewner equation
\begin{equation}  \label{radialL}
\partial_t h_t(z) =  a \, \cot(h_t(z) + B_t) ,  \;\;\; h_0(z)=z,
\end{equation}
where $B_t$ is a standard one-dimensional Brownian motion.  If we now define $g_t(e^{2iz}) =
\exp\{2ih_t(z)\}$, then the maps $g_t$ are those of radial SLE$_\kappa$ in
$\Disk$ with the parametrization chosen so that $g_t'(0) = e^{2at}$.
We write $\tilde h_t(z) = h_t(z) + B_t$, so that
\begin{equation}\label{RadialLEinHstar}
\dd\tilde h_t(z) = a \, \cot(\tilde h_t(z)) \di t + \dd B_t.
\end{equation}
By taking the spatial derivative of~\eref{RadialLEinHstar}
we obtain
\begin{equation}\label{eqnforhprime}
\tilde h_t'(z)  = \exp \left\{-a \int_0^t \csc^2 \tilde h_s(z) \di s \right\},
\end{equation}
and so
\[ \partial_t |\tilde h_t'(z)|^{2-d}
= \left(\frac{1}{4} - a\right) \rho(z)  |\tilde h_t'(z)|^{2-d} \]
where
$\rho(z) = \rho(x,y)$
is given by
\begin{equation}\label{defnrhoHstar}
\rho(x,y)= \Re(\csc^2 z) = \frac{\sin^2 x \cosh^2 y - \cos^2x \sinh^2y}{(\sin^2 x + \sinh^2 y)^2}  .
\end{equation}
If we fix $z$ and write
\[   Z_t = \tilde h_t(z) = X_t + i Y_t,\]
then $X_t=X_t(z)$ and $Y_t=Y_t(z)$ satisfy
\begin{equation}  \label{hstar}
 \dd X_t = a\, v_t  \di t + \dd B_t ,
 \; \;\;\; \partial_t Y_t = - a \, u_t  ,
\end{equation}
where
\[u_t = u(X_t,Y_t)= \frac{\sinh Y_t \cosh Y_t}  {|\sin Z_t |^2}, \;\;\;\; v_t= v(X_t,Y_t)=\frac{\sin X_t \cos X_t } {|\sin Z_t|^2}.\]
Furthermore, let $u(z) = u_0$ and $v(z) = v_0$ so that
\begin{equation}\label{defnuvHstar}
 u(z) = u(x,y) =\frac{\sinh y \cosh y} {\sin^2 x + \sinh^2 y}, \;\;\;\;
v(z) = v(x,y)  =\frac{\sin x\cos x}{\sin^2 x + \sinh^2 y}.
\end{equation}
Note that $\rho(x,y) = -v_x(x,y)$. Recall from~\eref{GFinHstar}  that if $
G_\Disk(z) = G_{\Disk}(z;1)$ denotes the Green's function for radial SLE$_\kappa$ in $\Disk$, then the Green's function for radial SLE$_\kappa$ in $\Half^*$ is
\[G(z) = G(x+iy) = 2^{2-d}e^{-2(2-d)y}G_{\Disk}(e^{-2y+i2x}).\]
 It\^o's formula (at $t=0$) tells us that if $G(z) = G(x,y)$ satisfies the partial differential equation
\begin{equation}\label{thePDEinHstar1}
\frac{1}{2}G_{xx} + a vG_x- a uG_y +  \left(\frac{1}{4} - a\right) \rho G = 0,
\end{equation}
then $M_t = M_t(z)$ defined by
\begin{equation}\label{martingaleM}
M_t =   |\tilde h_t'(z)|^{2-d}G(\tilde h_t(z))
\end{equation}
will be a local martingale.
  Even though~\eref{thePDEinHstar1} is just a coordinate transformation of~\eref{thePDEinDisk3}, it   is easier to analyze.

\begin{lemma}\label{LemHinHstar}
If  $p = (4a-1)-2(2-d)$, $\zeta = (4a-1)-(2-d)$,
and the function $H=H(z)=H(x,y)$ is defined as
$$H(z) = |\sin z|^p \,u(z)^\zeta$$
where $u$ is given by~\eref{defnuvHstar},
then $H$ satisfies the differential equation
\begin{equation}\label{PDEforHinHstar}
 \frac{1}{2}H_{xx} + a vH_x -auH_y+  \left(\frac{1}{4} - a\right) \rho H +ap H =0
\end{equation}
where
$v$ is given by~\eref{defnuvHstar} and $\rho=-v_x$ as in~\eref{defnrhoHstar}.
In particular,
\begin{equation}\label{martingaleN}
N_t = N_t(z) =  e^{ap t} |\tilde h_t'(z)|^{2-d}H(\tilde h_t(z))
\end{equation}
is a local martingale satisfying
$\dd N_t = (1-4a)\, v_t N_t\di B_t$.
\end{lemma}

\begin{proof}
This is a straightforward, although tedious,
computation using It\^o's formula.
\end{proof}

We will now write the local martingale $N_t$ 
 in a way which clearly indicates the analogy with the chordal case.
Let $p=(4a-1)-2(2-d)$ be as in the lemma and let $q$ be as in~\eref{defnq}. Set
\[\beta = ap =  -2aq=
  \frac{(2a-1)\, (4a-1)}{2}
= \frac{(4-\kappa) \, (8-\kappa)}{2\kappa^2},\]
and set
\begin{equation}\label{DefnUpsilonLambda}
 \cradt_t =  \Upsilon_{\Half^* \setminus \gamma_t}(z)
=  \frac{\sinh Y_t \cosh Y_t}{|\tilde h_t'(z)|} ,
\end{equation}
\begin{equation}\label{DefnUpsilonLambda2}
\Lambda_t  =\Lambda_t(z)
 = |\sin Z_t| \, u_t =  \frac{\sinh Y_t \cosh Y_t}{|\sin Z_t|},
 \end{equation}
so that
\begin{equation}\label{locmgN}
N_t = N_t(z) = e^{ap t} |\tilde h_t'(z)|^{2-d}H(\tilde h_t(z)) = e^{\beta t}\, \cradt_t^{d-2}  \, \Lambda_t^{4a-1}
\end{equation}
is a local martingale satisfying
\[    \dd N_t = (1-4a) \, v_t\, N_t \di B_t. \]
This is not a martingale because it ``blows up'' on the (measure zero) set of paths
which reach $z$. This local martingale also appears in \cite{KMak}.

We now use the Girsanov theorem to consider a new measure $\Pr^*$
on paths tilted by $N_t$.
(See~\cite{LawlerMultiFractAnalysis} for an
 explanation of what is meant by using the Girsanov theorem to tilt (or weight) by a local martingale.)
The Girsanov theorem states that
\[   \dd B_t = (1-4a) \, v_t \di t + \dd W_t, \]
where $W_t$ is a standard Brownian motion with respect to the new measure $\Pr^*$.
In particular,
\begin{equation}  \label{hstarweight}
  \dd X_t = (1-3a)\, v_t \di t + \dd W_t,
  \;\;\;\;  \partial_t Y_t = - a \, u_t  .
\end{equation}
Finally, recalling that $g_t'(0) = e^{2at}$,
let
\begin{equation}  \label{defnPhi}
 \Phi(z) = \E^*\left[g_T'(0)^{q}
  \right] =\E^*[e^{2aqT}]
    =\E^*[e^{-\beta T}],
    \end{equation}
where
\[ T=  T(z)  = \inf\{t: Z_t(z) = 0 \}, \]
and $\E^*$ denotes expectation with respect to $\Pr^*$. The next lemma shows that $\Phi$ is well defined.

\begin{lemma} \label{radGFlem2}
If $z \in \Half^*$, then $\Pr^*\{T < \infty\} =1$ and $\Phi(z) < \infty$.
Moreover,
\[           \lim_{z \rightarrow 0} \Phi(z) = 1.\]
\end{lemma}

\begin{proof}
First we show the following estimate which is valid for all $\kappa < 8:$ for every $\epsilon > 0$,
 there exists $\delta > 0$ such that if $|z| \leq \delta$, then
 $\Pr^*\{T > \epsilon\} \leq \epsilon$. To see this note that if $|Z_t|$ is small, then
\[   v_t = \frac{X_t}{X_t^2 + Y_t^2} \, [1 + O(|Z_t|^2)], \;\;\; 
\; u_t = \frac{Y_t}{X_t^2 + Y_t^2} \, [1 + O(|Z_t|^2)].\]
Hence the Loewner equation~\eref{hstarweight} near $0$ looks like
\begin{equation}  \label{chord2}
  \dd X_t = (1-3a) \, \frac{X_t}{X_t^2 + Y_t^2} \di t
    + \dd B_t ,\;
\;\;\; \partial_t  Y_t = - a \,\frac{Y_t}{X_t^2 + Y_t^2} .
    \end{equation}
 These are the equations for two-sided radial SLE$_\kappa$ (see, e.g., \cite{LawWerGF} for definition), and it
 is  known that this reaches the origin in finite time.  A scaling argument  gives the
 corresponding result for two-sided radial, and
 we can compare to get our estimate. In the $\beta \ge 0$ case, i.e., $\kappa \leq 4$, this immediately shows that $\Phi(z) \to 1$ as $z \to 0$. We also trivially have $\Phi(z) < 1$ in this case, so the lemma is complete for $\beta > 0$. For the remainder we assume that
 $\xi = -\beta > 0$, i.e., $4 < \kappa < 8$.

 We return to the equations~\eref{hstarweight} and let
 \[  \sigma = \inf\{t:X_t = 0\}. \]
 We claim there exists $c_1$ such that
$\E^*[e^{\xi \sigma}]  \leq c_1 $ for all $z$.  To see this, we note that $X_t$ reaches zero
  before $R_t$ where $R_t$ satisfies
    \[   \dd R_t = r \, \cot R_t \di t + \dd B_t, \;\;\;\;
   r = \max\{0,1-3a\},\]
since the latter equation has a larger drift away from the
origin.  Hence it suffices to show that $\E^*[e^{\xi \tau}]
< \infty$ where $R_0 = \pi/2$ and $\tau = \inf\{t: R_t = 0\}$.
This is true provided that $\xi < \lambda$ where $\lambda$ is
the smallest eigenvalue  of
\[       \frac 12 \, \phi''(x) + r \, \cot(x) \, \phi'(x)
 + \lambda \, \phi(x) =0. \]
 This eigenvalue corresponds to the positive eigenfunction $\phi$ and
 one can verify easily that
 \[      \phi(x) = [\sin x]^{1-2r} , \;\;\;\;
     \lambda =  \frac 12 - r > \xi .\]

We now use the previous paragraph for the argument that $\Phi(z) \to 1$ as $z \to 0$. We use it to show that for every $r >0$, there
  exists $\epsilon_r >0$
 such that if $|z| \leq \epsilon_r$
  then $\E^*[e^{\xi T}]
  \leq 1 + r$.  To see this, we consider excursions.
  Let $0 < \delta   \leq 1/2 $ with $e^{\xi \delta}
   + 2 \delta \leq 1 + r  $ and
  let $\tau_{-1} = 0$, $\rho_0=0$,
 \[   \tau_0 = T \wedge
  \inf\{t: t=\delta \mbox{ or } |Z_t| = 2 \delta \}, \]
 and if $\tau_0 < T$,
 \[ \rho_{j+1} = \inf\{t > \tau_j:X_t = 0 \}, \]
 \[ \tau_{j+1} = T \wedge \inf\{t > \rho_{j+1}: t = \rho_{j+1}
 + \delta \mbox{ or }  |Z_t| = 2 \delta \}. \]
 Then,
 \[   \E^*[e^{\xi T} \mid\tau_{j-1} \leq
  T < \tau_j] \leq e^{\xi \delta (j+1)}
  \, c_1^{j}. \]
 We can find $\epsilon > 0$ such that if $|z| < \epsilon$, then
 \[   \Pr^*\{\tau_0 < T\} \leq  \delta /(2c_1), \]
 and hence, since $Y_t$ decreases, for $\delta$
 sufficiently small, if $j \geq 1$, then
 \[  \E^*[e^{\xi T} ; \, \tau_{j-1} \leq
  T < \tau_j] = \Pr^*\{\tau_{j-1} \leq
  T < \tau_j\} \,  \E^*[e^{\xi T} \mid\tau_{j-1} \leq
  T < \tau_j] \leq \delta^j . \]
 Therefore,
  $\E^* \left[e^{\xi T}; \, T > \tau_0\right]
 \leq 2\, \delta$. Since $T \leq \delta$ on the event $\{T  \leq \tau_0\}$,
 we conclude that
 \[ \E^* \left[e^{\xi T}\right] \leq e^{\xi \delta} + 2 \delta
  \leq 1 + r ,\]
verifying that $\Phi(z) \to 1$ as $z \to 0$.

  Finally we use the last estimate to show that $\Phi(z) < \infty$ for all $z$. Combining the estimates of the last two paragraphs
  with the Markov property, we see that
  if $\Im(z) \leq \epsilon_r$, then $\E^*[e^{\xi T}] \leq
   (1+r) \, c_1$.  Using the deterministic estimate
   \[    - a \, \coth Y_t \leq \partial_t Y_t \leq -
   a \, \tanh Y_t,\] we see that if $z = x+iy$, then
   \[    \cosh^{-1} [(\cosh y) e^{-at}]
   \leq   Y_t \leq \sinh^{-1} [(\sinh y) e^{-at}]. \]
  In particular, there exists a deterministic function $\psi(y)$ such that
  $\inf\{t: Y_t = \epsilon_r\} \leq \psi(y)$ for $z = x+iy$, and hence
  \[   \Phi(x+iy) \leq e^{\beta \, \psi(y)} \, (1+r) \, c_1. \] 
This completes the proof. \end{proof}

We are now able to state the formula for the Green's function for radial SLE$_\kappa$ in $\Half^*$ as given by~\eref{FormulaForGFinHstar}.

\begin{theorem}\label{main_thm}
Suppose that $p = (4a-1)-(2-d)$, $\zeta = (4a-1)-2(2-d)$,  $\beta=ap$, and  $H(z) = |\sin z|^p u(z)^\zeta$ where $u(z)$ is given by~\eref{defnuvHstar}. The Green's function for radial SLE in $\Half^*$ is 
\[            G(z) = H(z) \, \Phi(z) , \]
where
\[           \Phi (z) = \E^*[e^{-\beta T}] \]
as in~\eref{defnPhi}.  In other words,
 If $z
\in \Half^*$, then for $\cradt_{\infty}(z)$ defined by~\eref{DefnUpsilonLambda},
\[   \lim_{\epsilon \downarrow 0}
\epsilon^{d-2}\,
 \Pr\{\cradt_{\infty}(z) \leq \epsilon\}
 = c^* \, G(z),   \;\;\; {\it  where} \;\;\;
 c^* =  2 \, \left[\int_0^\pi \sin^{ 8 /\kappa}  x
 \di x \right]^{-1}.
\]
If $\kappa = 4$, then $G(z) = H(z)$.
\end{theorem}

 We end this section with a brief discussion of some of the ingredients that go into the proof.
 For small times, radial SLE in $\Half^*$ is very close
 to chordal SLE$_\kappa$ in $\Half$.  Let us make a precise
 statement of this fact.  For $ 0 < r \leq 1$,
 let
 \[   \lambda_r = \inf\{t: |\gamma(t)| = r\}. \]
The half disk of radius $1$ about $0$ in $\Half^*$ can be
considered as the half disk in $\Half$, and hence we can
view radial SLE in $\Half^*$ up to time $\lambda_r$ as
living in $\Half$.  Let $H_r$ be the unbounded
component of $\Half \setminus \gamma_{\lambda_r}$,
and let $\crad_{H_r}(z)$, $\cradt_{H_r}(z)$ denote
 ($1/2$ times) the conformal radius about $z$
where $H_r$ is considered a subdomain of $\Half$
and $\Half^*$, respectively. If $|z| \leq r$,
\begin{equation}
\label{jun14.1}
 \Upsilon_{H_r}(z) = \cradt_{H_r}(z)
 \, [1 + o_r(1)].
 \end{equation}
 Here and for the remainder of this section,
 $o_r(1)$ will denote a term that goes to zero in
 $r$ uniformly over all other variables.

 Let us assume that $\gamma$
is parametrized using the half plane capacity in
$\Half$; that is, the maps $g_t$ satisfy
\eref{chordalL}.  Since the half disk of radius $r$
has half plane capacity $r^2$, we know that
$\lambda_r \leq r^2/a$.
Considered as a solution to a radial equation~\eref{radialL}, the parameterization is slightly
different, say $\tilde \lambda(t) = \lambda(\sigma(t)).$
However,
\[   \sigma(t) = t \, [1+O(t) ].\]
In particular, in the radial parametrization,
$\lambda_r \leq (r^2/a) \, [1+O(r)].$

For $r \leq 1$, let $\Pr$ and $\tilde \Pr$ denote the probability measures
on $\gamma(s)$,  $0 \leq s \leq \lambda_r$, given by chordal and
 radial SLE$_\kappa$ respectively.  We can
   think of these either as measures on curves
   modulo reparametrization or on parametrized
   curves; however, in the latter case, one must
   take the same parametrization (chordal or radial)
   in both cases.  Then $\Pr$ and $\tilde \Pr$
   are mutually absolutely continuous with
   $   \tilde \Pr = \Pr [1 + o_r(1)] $;
   that is,
   \[      \left| \frac{\dd \tilde \Pr}
    {\dd\Pr} - 1\right| = o_r(1) .\]

 The analysis of the chordal case shows that
   if $0 < \epsilon
   \leq 1/2$, then for $|z| \leq r/\log(1/r)$,
   \[     \Prob{\Upsilon_\infty(z) \leq \epsilon
    \, \Upsilon_0(z)} = \Prob{\Upsilon_{H
    _r}(z) \leq \epsilon \, \Upsilon_0(z)}
     \, [1 + o_r(1)]. \]
   Also, there exists $\alpha > 0$ such that
    \[     \Prob{\Upsilon_\infty(z) \leq \epsilon
    \, \Upsilon_0(z)} = c^* \, S(z)^{4a-1}
     \, \epsilon^{4a-1} \, [1 + O(\epsilon^\alpha)],\]
  where in this case the $O(\cdot)$ is uniform
  over $z$ and we recall $S(z) = \sin (\arg z)$.  Therefore, there exists $\epsilon_r
  \downarrow 0$, such that for $\epsilon < \epsilon_r$,
  \[ \tilde \Pr\{\Upsilon_{H_r}(z) \leq \epsilon \, \Upsilon_0(z)\}
     =  c^* \, S(z)^{4a-1}
     \, \epsilon^{4a-1} \, [1
      + o_r(1)], \]
     and hence by~\eref{jun14.1}, if
     $|z| \leq r/\log(1/r)$,  $\epsilon < \epsilon_r$, then
  \[  \tilde
   \Pr\{\cradt_{H
    _r}(z) \leq \epsilon \, \cradt_0(z)\}
     =  c^* \, \Lambda(z)^{4a-1}
     \, \epsilon^{4a-1} \, [1
      + o_r(1)],\]
 where $\Lambda(z)$ is given by~\eref{defnlambda}. It is not obvious at the moment, and we prove
 in Proposition~\ref{jun14.prop1}, that
 if $|z| \leq r/\log(1/r)$, then
 \[  \tilde
   \Pr\{\cradt_{H_r}(z) \leq \epsilon \, \cradt_0(z)\}
      =  \tilde
   \Pr\{\crad_{\Half
    \setminus
    \gamma_\infty}(z) \leq \epsilon \, \cradt_0(z)\}
     \,  [1
      + o_r(1)].\]

 We will not give the details of the proof
 of the estimates in the last
 paragraph, but we will sketch it here.  One can compare
chordal SLE$_\kappa$ from $0$ to $\infty$ in $\Half$ and radial
SLE  from $0$ to $i$ in $\Half$ by tilting by a particular local
martingale; see~\cite{ParkCity}.  Radial SLE$_\kappa$ in $\Half^*$
can be obtained from radial SLE$_\kappa$ in $\Half$ from $0$ to $i$
by the (multiple valued) transformation $f$ as in~\eref{jun5.1}.
 In both cases, one sees that the driving function changes from
 a standard Brownian motion to one with a drift.  Under our conditions,
 the drift is uniformly bounded, and since time is bounded
 by $O(r^{1/2})$, we can bound the Radon-Nikodym
 derivative in the  change of measure.

\section{Proof of Theorem \ref{main_thm}} \label{SectMainThm}

We start by reviewing the proof of the existence of the
Green's function in the chordal case because we will need to use
some of these facts.
Fix $z \in \Half$ and
let $\gamma$ be a chordal SLE$_\kappa$
path from $0$ to $\infty$ in $\Half$
which has equations
\begin{equation}  \label{chor2}
\dd X_t = a \, \frac{X_t}{X_t^2 + Y_t^2} \di t + \dd B_t, \; \;\;\;
  \dd Y_t = - a \, \frac{Y_t}{X_t^2 + Y_t^2} \di t .
\end{equation}
The corresponding local martingale is
\[  M_t = M_t(z) =
  \Upsilon_{t} ^{d-2}  \, S_t^{4a-1} , \]
where
\[  \Upsilon_t = \Upsilon_{\Half \setminus \gamma_t}(z),  \;\;\; \;
 S_t =  S_{\Half \setminus \gamma_t}(z; \gamma(t), \infty) \]
 as in Section~\ref{SectChordalGFReview}.
  Let $\tau_\epsilon = \inf\{t: \Upsilon_t = \epsilon\, \Upsilon_0 \}$.  We will now review the proof that
\begin{equation}  \label{jun14.4}
  \lim_{\epsilon \downarrow 0}
   \epsilon^{d-2} \, \Prob{\tau_\epsilon < \infty}
     = c^* \, S(z)^{4a-1}.
     \end{equation}
We begin by observing that
\begin{eqnarray*}
\Prob{\tau_\epsilon < \infty} &  = &
  \E\left[1\{\tau_\epsilon < \infty \}\right]\\
  &  =  &(\epsilon \Upsilon_0)^{2-d} \, \E\left[ M_{\tau_\epsilon} \, S_
  {\tau_\epsilon} ^{1-4a} ; \, \tau_\epsilon
   < \infty\right]  \\
& = &  \epsilon^{2-d} \, S(z)^{4a-1}
 \, \cGF(z)^{-1} \, \E\left[ M_{\tau_\epsilon} \, S_
  {\tau_\epsilon} ^{1-4a} ; \, \tau_\epsilon
   < \infty\right]  \\
   &  =  & \epsilon^{2-d} \,  S(z)^{4a-1} \,
          \E^*\left[S_{\tau_\epsilon}^{1-4a}\right]
           .
           \end{eqnarray*}
 Here we write $\E^*$ for the measure obtained
 by tilting by the local martingale $M_t$.  This shows that proving~\eref{jun14.4} is equivalent to
 showing that
 \[   \lim_{\epsilon \downarrow 0}
        \E^*\left[S_{\tau_\epsilon}^{1-4a}\right]
        = c^*.\]
This is proved using a one-dimensional
stochastic differential equation.
If the path is reparametrized so that $\Upsilon_{\sigma(t)} = e^{-2at}\, \Upsilon_0$
and $\Theta_t = \arg Z_{\sigma(t)}$, then $\Theta_t$ satisfies
\begin{equation}  \label{may16.3}
    \dd\Theta_t = (1-2a) \, \cot\Theta_t \di t + \dd B_t,
    \end{equation}
where $B_t$ denotes a standard one-dimensional Brownian motion.  Note that $1-2a < 1/2$
and by comparison with a Bessel process we see that the process
reaches $0$ or $\pi$ in finite time.  This reflects the fact that
$\Upsilon_\infty < \infty$; indeed, the time 
in the new parametrization at
which the  boundary is reached
 is $- (2a)^{-1}\log \Upsilon_\infty$. If
 $\hat M_t = M_{\sigma(t)}$, then
 \[  d \hat M_t = 4a \, [\cot \Theta _t] \, \hat{M}_t
 \, dB_t.\]
If we weight by the local martingale $M_t$
(or, equivalently, by $\hat M_t$), then we get
\begin{equation}  \label{may16.4}
         \dd\Theta_t = 2a \, \cot \Theta_t \di t + \dd W_t,
         \end{equation}
where $W_t$ is a standard Brownian motion in the new measure
$\Pr ^*$.  This weighted measure is sometimes called two-sided
radial SLE$_\kappa$ from $0$ to $\infty$ through $z$ (stopped
at $T = T_z$, the time it reaches $z$). Since $2a > 1/2$, we see
that in the new measure the process survives forever; in other
words, the weighted process hits $z$.  The invariant probability
density for the process~\eref{may16.4} is
\[   f(\theta) =\frac{c^*}{2} \,   \sin^{4a} \theta, \;\; 0\le \theta < \pi, \;\;\; {\rm where} \;\;\;
   c^* =2\,  \left[\int_0^\pi \sin^{4a}x \di x \right]^{-1}.\]
Moreover, there exists $\lambda > 0$ such
that  if $p^*_t(\theta_1,\theta_2)$ denotes the transition density
for the process, then for every $t$ with $e^{-2at} \leq 1/2$,
\[   p_t^*(\theta_1,\theta_2) = f(\theta_2) \, \left[1 + O(e^{-\lambda t})
 \right], \]
where the error term $O(e^{-\lambda t})$ is uniform over $\theta_1$, $\theta_2$.

 The proof in $\Half^*$ is similar  but  with some
 added complications.  Recall from~\eref{locmgN} that the local martingale is
 \[   N_t = N_t(z) =  e^{\beta t} \cradt_{t}^{d-2} \, \Lambda_t^{4a-1}
  \]
 where $\cradt_t =\crad_{\Half^*\setminus \gamma_t}(z)$ and $\Lambda_t = \Lambda_t(z)$ as in~\eref{DefnUpsilonLambda}.
  We now fix $z \in \Half^*$ and allow constants in what follows  to depend on $z$.
For each $\epsilon  > 0$, let
 \[  \tau_\epsilon = \inf\{t: \cradt_{t}
   = \epsilon \, \cradt_0 \}.    \]
 What we need to show is that
 \[  \Prob{\tau_\epsilon < \infty}
  = c^* \, \Phi(z) \, \Lambda(z)^{4a-1} \, \epsilon^{2-d}
   \, [ 1 + o(1)], \]
   where $o(1)$ represents a term which
   may depend on $z$ that goes
   to zero as $\epsilon \rightarrow 0$.
In analogy with the chordal case, observe that
  \begin{eqnarray*}
\Prob{\tau_\epsilon < \infty}&
  =  &\E\left[  1\{\tau_\epsilon < \infty\}
 \right]\\
 & = & \epsilon^{2-d} \,\cradt_0^{2-d}
  \,  \E\left[N_{\tau_\epsilon}
  \, e^{-\beta \tau_\epsilon} \, \Lambda_{\tau_\epsilon}
  ^{1-4a} ; \,\tau_\epsilon < \infty\right]\\
  & = & \epsilon^{2-d} \,\cradt_0^{2-d}\, N_0 \, \E^*\left[
      e^{-\beta \tau_\epsilon} \,
      \Lambda_{\tau_\epsilon}
  ^{1-4a} \right]\\
  & = & \Lambda(z)^{4a-1} \, \epsilon^{2-d}
   \, \E^*\left[
      e^{-\beta \tau_\epsilon} \,
      \Lambda_{\tau_\epsilon}
  ^{1-4a} \right],
  \end{eqnarray*}
where $\E^*$ denotes the measure obtained by
tilting by the local martingale $N_t$.
To prove the theorem, we need to show that
\[  \lim_{\epsilon \downarrow 0}
  \E^*\left[
      e^{-\beta \tau_\epsilon} \,
      \Lambda_{\tau_\epsilon}
  ^{1-4a} \right] = c^* \,\E^*\left[
   e^{-\beta T} \right] = c^* \, \Phi(z).\]
The basic reason this is true can be explained
as follows.  As $\epsilon \rightarrow 0$, the
quantity $\Lambda_{\tau_\epsilon}$ is very much
like $S_{\tau_\epsilon}$ and asymptotically it
behaves as in the chordal case.  The quantity
fluctuates as $\epsilon \rightarrow 0$, but
with a stationary distribution giving
\[    \E\left[ \Lambda_{\tau_\epsilon}
  ^{1-4a} \right] \rightarrow c^* . \]
The quantity $e^{-\beta \tau_\epsilon}$
is a ``macroscopic'' quantity that does not change
much when $\epsilon$ is small.  Since
 $\Lambda_{\tau_\epsilon} $ {\em is} fluctuating
 on the small scale, we can see that the
 two random variables are asymptotically independent,
 \[     \E^*\left[
      e^{-\beta \tau_\epsilon} \,
      \Lambda_{\tau_\epsilon}
  ^{1-4a} \right]  \sim
   \E^*\left[
      e^{-\beta \tau_\epsilon}\right] \,
      \E^*\left[  \Lambda_{\tau_\epsilon}
  ^{1-4a} \right] \rightarrow
  c^* \,  \E^*[e^{-\beta T}].\]
 
To prove this, we first note
that  if $V$ is an event, then
\begin{equation}  \label{jun13.3}
 \Pr[\{\tau_\epsilon < \infty\}
 \cap V ] =
\Lambda(z)^{4a-1} \, \epsilon^{2-d}
   \, \E^*\left[
      e^{-\beta \tau_\epsilon} \,
      \Lambda_{\tau_\epsilon}
 ^{1-4a}\, 1_V \right].
  \end{equation}
Let
\[  \rho =
\rho_{\epsilon} = \inf\left\{t: \cradt_t
 = \epsilon^{1/2} \right\}.\]
 For $t \geq \rho$, let $
 \hat \gamma(t) = \hat \gamma^\epsilon(t) =
\tilde h_{\rho}(\gamma(t))$,
\[  \lambda =   \lambda_{\epsilon}
 = \inf\left\{t \geq \rho: |\hat \gamma(t)| =
    \epsilon^{1/8} \right\}, \]
 and let $ V_{\epsilon}$ be
 the event
 \[   V_{\epsilon} =\left\{|\tilde h_{\rho}(z)| \leq
                    \epsilon^{1/3}, \lambda
                  > \tau_\epsilon  \right\}.\]
To prove the theorem, it suffices to show that
\begin{equation}  \label{jun13.2}
    \lim_{\epsilon \downarrow 0}
 \epsilon^{d-2}
  \,  { \Pr[\{ \tau_\epsilon
    < \infty \}  \setminus V_\epsilon]
    } = 0 ,
\end{equation}
 and
 \begin{equation}  \label{jun10}
    \Pr[\{\tau_\epsilon < \infty\}
    \cap V_\epsilon]
    \sim c^* \, \Phi(z) \,\Lambda(z)^{4a-1}
     \, \epsilon^{2-d}.
    \end{equation}
     To prove~\eref{jun13.2}, let
 \[  \sigma = \sigma_\epsilon
 = \inf\{t: |\gamma(t) - z|
   \leq \epsilon^{1/2}/10\}.\]
The Koebe one-quarter  estimate
implies that
 $c_1 \epsilon^{1/2} \leq \Upsilon_{\sigma} \leq
\epsilon^{1/2}. $  By considering the image of the
line segment from $z$ to $\gamma(\sigma)$
under the map $\tilde h_{\sigma}$ and using
the Beurling estimate, we see that
there exists $c_2$ such that
$     |\tilde h_{\sigma}(z)|
   \leq c_2 \, \epsilon^{1/4}. $
Since
\[    |\tilde h_{\rho}(z)| =
       \frac{ \Im [\tilde h_{\rho}(z)]}
         {S_\rho(z)} \leq \frac{ \Im [\tilde h_{\sigma
         }(z)]}
         {S_\rho(z)}
          = \frac{ S_\sigma(z) \, |\tilde h_\sigma
          (z)|}
         {S_\rho(z)} ,\]
then $|\tilde h_{\rho}(z)| \geq
 \epsilon^{1/3}   $
implies that
\[       S_\sigma
(z)  \geq c_2^{-1} \,  S_\rho(z) \, \epsilon^{1/12}
.\]
Using Proposition~\ref{prop22} below, we see that there exists $c_3 < \infty$
such that
\[ \Prob{\tau_\epsilon < \infty, \,
   S_\sigma(z) \geq c_2^{-1} 
   S_\rho(z) \, \log(1/\epsilon) \,\mid\,
   \gamma_\sigma
    }   \]
    \begin{equation*}
  \hspace{1in} \leq c_3 \,\epsilon^{(1-4a)/12}\,
    \Prob{\tau_\epsilon < \infty \mid
   \gamma_\sigma}. 
   \end{equation*}
Therefore,
\[  \Prob{\tau_\epsilon < \infty, \,
|\tilde h_\rho(z)| \geq  \epsilon^{1/3}}
  \leq c_3 \,\epsilon^{(1-4a)/12}
 \,  \Prob{\tau_\epsilon < \infty}
.\]
 It follows from Proposition~\ref{jun14.prop1}
below that
    \[  \Prob{\lambda <
    \tau_\epsilon < \infty \,\mid\,
|\tilde h_\rho(z)| \leq \epsilon^{1/3} }
= o(\epsilon^{2-d}), \]
establishing~\eref{jun13.2}. 

As for showing~\eref{jun10}, from~\eref{jun13.3} we see that
we need to show that
\[  \lim_{\epsilon \downarrow 0}
  \E^*\left[
      e^{-\beta \tau_\epsilon} \,
      \Lambda_{\tau_\epsilon}
  ^{1-4a} 1_{V_\epsilon}\right] = c^* \, \E^*[e^{-\beta T}].\]
  On the event $V_{\epsilon}$, we have $\tau_\epsilon - \tau_\rho
  = o(1)$.  Hence, it suffices
  to show that
\[  \lim_{\epsilon \downarrow 0}
  \E^*\left[
      e^{-\beta \tau_\rho} \,
      \Lambda_{\tau_\epsilon}
  ^{1-4a} \, 1_{V_\epsilon}\right] = c^* \, \E^*[e^{-\beta T}].\]
The expectation on the left equals
 \[   \E^*\left[
      e^{-\beta \tau_\rho} \,
      \E^*[\Lambda_{\tau_\epsilon}
  ^{1-4a} \, 1_{V_{\epsilon} } \mid \gamma_\rho] \right].\]
We now use the fact that as $\epsilon \rightarrow 0$,
we get better and better approximation to the chordal
case.  In particular, since $\epsilon \ll
\epsilon^{1/2}$ and $\Pr(V_\epsilon) 
\rightarrow 1$,
\[  \E^*\left[\Lambda_{\tau_\epsilon}
  ^{1-4a} \, 1_{V_{\epsilon} } \,\mid\, \gamma_\rho\right] \sim
 \E^*[S_\tau^{1-4a}] , \]
where the right side is the
 $\E^*[S_\tau^{1-4a}]$ in the
calculation~\eref{jun13.3}.  This
gives
\[  \E^*\left[\Lambda_{\tau_\epsilon}
  (z)^{1-4a} \, 1_{V_{\epsilon} } \,\mid\, \gamma_\rho\right]  =
  c^* \,[1 + o(1)], \]
  and hence,
\[  \E^*\left[
      e^{-\beta \tau_\rho} \,
      \Lambda_{\tau_\epsilon}
  ^{1-4a}\, 1_{V_\epsilon}\right]  
  = c^* \,  \E^*\left[
      e^{-\beta \tau_\rho} \,
      1_{V_\epsilon}\right] 
 \,[1 + o(1)].\]
Finally, the dominated convergence
theorem implies that
\[ \lim_{\epsilon \downarrow 0}
 \E^*\left[
      e^{-\beta \tau_\rho} \,
      1_{V_\epsilon}\right] 
= \E^*\left[e^{-\beta T}\right]
 = \Phi(z) .\]
 
\subsection{Some lemmas about radial SLE}

Here we discuss the results needed to prove~\eref{jun13.2}.  We state the following lemma
proved in~\cite{contpaper} which is a radial SLE
analogue of a chordal SLE estimate proved in~\cite{AlbKoz}.

\begin{lemma}   \label{contlemma}
 There exist $0 < c_1 < c_2
< \infty$
such that if $-\pi/2 \leq x \leq \pi/2$ and
$r < |x|$, then if $\gamma$ is radial SLE
in $\Half^*$,
\[    c_1 \,(r/|x|)^{4a-1}
   \leq \Prob{\dist(x,\gamma) \leq r }
     \leq c_2 \, (r/|x|)^{4a-1}. \]
\end{lemma}

We give a corollary of this. Suppose $z=x+iy \in \Half^*$ with $y \leq 1$.  For every $\rho >0$,
there exists $K$ such that if $\dist(z,\gamma)
\geq Ky$, then $\cradt_{\infty}
(z) \geq (1-\rho) \, \cradt_0(z)$.
This gives the following.  Recall
that $S(x+iy) \asymp \Lambda(x+iy)$ for
$y \leq 1$.

\begin{lemma} \label{contlemma2}
 For every $\rho > 0$, there exists
$c = c_\rho < \infty$ such that if $z = x+iy \in \Half^*$ with $y > 0$, then
\[      \Pr\{\cradt_\infty(z) \leq
   (1-\rho) \,\cradt_0(z) \} \leq c \, S(z)^{4a-1}, \]
   where $S(z) = y/|z|$.
\end{lemma}

\begin{proposition} \label{prop22}
 There exists $0 < c_1
< c_2 < \infty$ such
that if $z = x+iy \in \Half^*$ with $y \leq 1$ and
$\epsilon   \leq 1/2$, then
\[   c_1 \, S(z)^{4a-1} \, \epsilon^{2-d}
\leq \Pr\{\cradt_\infty(z) \leq \epsilon \,
\cradt_0(z)\}
    \leq  c_2 \, S(z)^{4a-1} \, \epsilon^{2-d} \]
    where $S(z) = y/|z|$.
    \end{proposition}

  \begin{proof}
  The lower bound for $|z| <  1/4$ can be established
  by comparison with chordal SLE$_\kappa$.  For
  other $z$, we can reduce to the $|z| < 1/4$ case by noting that there exists $c >0$
  (independent of $z$)
  such that  with probability at least $c$,
    $Z_t  \leq 1/8$
  and $S_t > c \, S(z)$
  for some $t$.  We will prove the harder
   upper bound.

   Recall that $S(z) \asymp \Lambda(z)$,
  so we can use either on the right  side (with appropriate
  constant).  Without loss of generality we assume
  $\epsilon = 2^{-n}$ for positive integer $n$.
The case $\epsilon = 1/2$ was proved in~\cite{contpaper}.
As before, let
\[  \tau_\epsilon = \inf\{t:
  \cradt_t = \epsilon \, \cradt_0\}.\]

  We will first  find a $\delta$ such that the estimate
  holds for $|z| \leq \delta$.
  Consider the
  following sequence of stopping times:
  \[   \sigma_0 = \xi_0  = 0 , \;\;\;\;
   \sigma_1 = \inf\{t: |\gamma(t)| = 2^{-1} \}. \]
  For $k \geq 1$, let
  \[   \eta^{k}(t) = \tilde h_{\sigma_{k}}(\gamma(t
     )), \;\;
  \xi_k = \inf\left\{t \geq \sigma_k:
  |\eta^k(t) -  \tilde h_{\sigma_k}(z)
  | \leq 2^{-m} \, \Im[ \tilde h_{\sigma_k}(z)] \right\}. \]
  Here $m \geq 2$ is a fixed  integer which we will specify
  below.  If $\xi_k = \infty$ we set $\xi_j = \sigma_j = \infty$
  for $j > k$.
    If $\xi_k < \infty$, let
   \[  \gamma^k(t) =  \tilde h_{\xi_{k}}(\gamma(t
     + \xi_k)), \;  z_k = \tilde h_{\xi_k}(z), \;\ \sigma_{k+1} = \inf \left \{ t \geq \xi_k : |\gamma^k(t)| = 2^{-1} \right \}. \]
    The Koebe one-quarter  theorem implies that
    \[  2^{-m-1} \cradt_{\sigma_k} \leq \cradt_{\xi_k}
     \leq \cradt_{\sigma_k}.\]

     By considering the conformal image under $\tilde h_{\xi_k}$ of the line segment
     of length $2^{-m} \, \Im(\tilde h_{\sigma_k}(z))$
     between $\tilde h_{\sigma_k}(z)$ and $\eta^k(\xi_k)$ and using
     the Beurling estimate,
     we see that we can choose $m$ (uniformly
     over $k$ and $z$) so
     that  $|z_k| \leq   \Im(z) \leq |z|$.  We fix such an $m$.
    Let
  $E_k$, $k=0,1,2,\ldots$, denote the event
  \[     E_k = \left\{\xi_k \leq \tau_\epsilon < \sigma_{k+1}
    \right\}.\]
   We claim that there exists $c_0 < \infty$, $\delta > 0$,
    such that if $|z| \leq \delta$, then
\begin{equation}  \label{jun8.3}
    \Pr[E_k] \leq c_0 \, 2^{k} \,
    \, S(z)^{4a-1} \,  |z|^{k\beta} \,\epsilon^{2-d}, \;\;\;\;
     \beta = 2a - \frac 12 > 0.
    \end{equation}
   The $k=0$ case (for any $\delta \leq 1$)
    follows from the chordal estimate and the
     absolute continuity of radial and chordal
    SLE$_\kappa$ up to time $\sigma_1$.  To
    be specific, let $c_0 < \infty$ be a constant such
    that for all $|z| \leq 1/8$,
 \begin{equation}  \label{jun7.1}
   \Prob{\tau_\epsilon < \sigma_1} \leq c_0 \, S(z)^{4a-1}
     \,\epsilon^{2-d}.
     \end{equation}

   Let us consider the $k=1$ case.  Suppose that $z=x+iy$,
   $r \leq 1$, and
   $ry
    < \cradt_{\sigma_1} \leq 2r y$.  If we consider
   $\gamma:= \gamma[0,\sigma_1]$ as a chordal path in $\Half$,
   then we claim that
   \begin{equation}  \label{jun8.1}
   S_{\sigma_{1}}(z) \leq c_1 \, r^{1/2} \, y.
   \end{equation}   To see this we consider $D = \Half
   \setminus \gamma $ and let $g:D \rightarrow
   \Half$ be a conformal transformation with $g(\gamma(\sigma_1))
   = 0$, $g(\infty) = \infty$. Since the angle is
   a conformal invariant,  $S $ is comparable
   to the minimum of the probabilities that a Brownian
   motion starting at $g(z)$ exits $\Half$ at the positive or negative
   real axis, respectively.
   By conformal invariance, this is bounded above
   by the probability that a Brownian motion starting at $z$
   reaches the circle of radius $1/2$ without leaving $D$.
   By the Beurling estimate, the probability that it gets to the circle
   of radius $2 y $ without leaving $D$ is  $O(r^{1/2})$ and given
   this the probability to reach the disk of radius $1/2$ is
   $O(y)$, provided that $\delta \leq 1/8$.  This
   establishes~\eref{jun8.1}.

Lemma~\ref{contlemma} implies that
   \[      \Pr\{\xi_1 < \infty \,\mid\, 2^{-j} y < \cradt_{\sigma_1}
    \leq 2^{1-j} y \} \leq c \, S_{\sigma_1}(z)^{4a-1}
      \leq c_2 \, 2^{-\beta j } \, y^{ 4a-1}
       , \]
       where $
       \beta = 2a - \frac 12 > 0.$    The $k =0$
       estimate implies that
  \[ \Pr\{   \cradt_{\sigma_1} \leq 2^{1-j}\, y\}
   \leq c \, 2^{j(d-2)}. \]
Hence,  we see that there exists $c_2$
such that
\begin{eqnarray*}
 \Pr\{\xi_1 < \infty,  \,2^{-j} y < \cradt_{\sigma_1}
     \leq    2^{1-j} y \}  &  \leq & c_2 \,
        2^{j(d-2)}\, \, 2^{-\beta j} \,
         y^{4a-1}\\
         &
           = & c_2 \,
        2^{j(d-2)}\, \, 2^{-\beta j} |z|^{4a-1} \, S(z)^{4a-1}.
        \end{eqnarray*}
        We now choose $\delta < 1/8$ sufficiently small
   so that
 \[      c_2 \delta^{\beta} 2^{m(2-d)} \sum_{j=1}^\infty
 2^{-\beta j} \leq \frac{c_0}{2} . \]
     Using~\eref{jun7.1}, we see that
\begin{equation}  \label{jun7.2}
\Pr\left[E_1 \,\mid\, \xi_1 < \infty,  \, 2^{-j} y < \cradt_{\sigma_1}
    \leq 2^{1-j} y\right] \leq c_0 \, 2^{(m+j-n)(2-d)} .
    \end{equation}
 By combining this with the last estimate and summing over
 $j$, we see that for $|z| \leq \delta$
\[   \hspace{-.4in}
\Pr[E_1] = \sum_{j=1}^\infty
  \Pr\left[E_1 \cap  \{2^{-j} y < \cradt_{\sigma_1}
    \leq 2^{1-j} y\} \right]\leq 2^{-1} \, c_0 \, |z|^{2a - \frac 12}
     \,  S(z)^{4a-1}
   \, \epsilon^{2-d}. \]

 The case $k > 1$ is done inductively in the same way.
 The values $m$, $\delta$, $c_0$ are fixed.
If~\eref{jun8.3} holds for $k-1$, then~\eref{jun7.2} becomes
\[ \Pr[E_k \,\mid\, \xi_1 < \infty,  \, 2^{-j} y < \cradt_{\sigma_1}
    \leq 2^{1-j} y] \leq 2^{1-k} \, c_0 \, 2^{(m+j-n)(2-d)}
 .\]
This establishes~\eref{jun8.3} and hence the proof is complete for $|z| \leq \delta$.

Now for $z$ with $\Im(z) \leq 1$ but $|z| > \delta$ let
 \[  \rho =
   \rho_r = \inf\{t: |z-\gamma(t)| \leq r \}. \]
 By considering the line segment from $z$ to $\gamma(\rho)$
 and using conformal invariance and the Beurling estimate,
 we see there exists $\rho = \rho_\delta$ such that  $|\tilde h_{\rho}(z)| \leq \delta$ for
 all $y \leq 1$.  We
 fix such an $r < 1/10$ and first note that
 Lemma~\ref{contlemma2}
 implies that
 \[      \Pr\{\rho < \infty\} \leq c \, S(z)^{4a-1}, \]
and the estimate above gives
\[   \Pr\{\tau < \infty \,\mid\, \rho < \infty\} \leq c \, \epsilon^{2-d}
  \]
 which finishes the proof.
    \end{proof}

   \begin{proposition}  \label{jun14.prop1}
     For every $\rho >0$, there
   exists $c < \infty$
  such that if $\epsilon >0$ and
   \[  \lambda = \lambda_r =\inf\{t: |\gamma(t)| = r\}, \]
   then for $|z| \leq r$,
  \[ \Pr\{\tau_\epsilon < \infty, \, \cradt_\infty \leq (1-\rho) \, \cradt_{\lambda}
 \} \leq  c\, S(z)^{4a-1} \,
 (|z|/r)^{4a-1}\, \epsilon^{2-d}. \]
    \end{proposition}

    \begin{proof}
 We fix $\rho$ and allow constants to depend
 on $\rho$.  Without loss of generality, we assume
 that $\epsilon = 2^{-n}$.
   Let $E_k$ denote the event
 \[   E_k = \{2^{-k} \, \cradt_0
  < \cradt_\lambda \leq 2^{-k+1} \, \cradt_0\}.\]
As shown above in the proof of Proposition \ref{prop22}, there exists $c$ such that
on the event $E_k$,
\[     S_\lambda(z) \leq  c \, 2^{-k/2} \,
     \Im(z)\, r^{-1}
      \leq c \, 2^{-k/2} \, (|z|/r), \]
and if $k > 1$,
\[  \Pr[E_k] \leq c \,S(z)^{4a-1} \,
 2^{-k(2-d)}. \]
There exists $c_1$ such that
\[
  \Pr\{\tau_\epsilon < \infty, \, \cradt_\infty
 \leq (1-\rho) \, \cradt_\lambda \,\mid\, E_k\}
    \leq  c_1 \, S_\lambda(z)^{4a-1}
    \, 2^{-(n-k)(2-d)}.
   \]
By combining these estimates and summing over $k$,
we see that there exists $c_2$ such that
\[ \Pr\{\tau_\epsilon < \infty,  \, \cradt_\infty \leq (1-\rho) \, \cradt_{\lambda}
 \} \leq c_2 \, (|z|/r)^{4a-1} \, S(z)^{4a-1}
  \, \epsilon^{2-d} \]
and the proof is complete.
  \end{proof}

 \subsection*{Acknowledgments}
 
This paper had its origin at the
\emph{Random Spatial Processes} program held  at
the Mathematical Sciences Research Institute in Berkeley, CA, during Spring 2012. All three authors benefitted from extended visits to MSRI and are grateful to the organizers for the invitation to participate in the program.
 Thanks are also owed to Dapeng Zhan, Nikolai Makarov, and Nam-Gyu Kang for useful discussions. The research of the second author is supported in part by the Natural Sciences and Engineering Research Council (NSERC) of Canada and the research of the third author is supported by National Science Foundation grant  DMS-0907143.


\begin{thebibliography}{}

\bibitem{AlbKoz}
T.~Alberts and M.~J.~Kozdron.
\newblock Intersection probabilities for a chordal SLE path and a semicircle
\newblock {\em Electron. Comm. Probab.}, 13:448--460, 2008.

\bibitem{Bef}
V.~Beffara.
\newblock The dimension of the SLE curves.
\newblock {\em Ann. Probab.} 36:1421--1452, 2008.

\bibitem{KMak}
N.-G.~Kang and N.~Makarov.
\newblock Radial SLE martingale-observables.
\newblock Preprint, 2012.

\bibitem{SLEbook}
G.~F.~Lawler.
\newblock {\em Conformally Invariant Processes in the Plane},
\newblock volume 114 of  {\em Mathematical Surveys and Monographs}.
\newblock American Mathematical Society, Providence, RI, 2005.

\bibitem{LawlerMultiFractAnalysis}
G.~F.~Lawler.
\newblock  Multifractal analysis of the reverse flow for the Schramm-Loewner evolution.
\newblock In C.~Bandt, P.~M\"orters, and M.~Z\"ahle, editors,
\newblock {\em Fractal Geometry and Stochastics IV},
\newblock volume 61 of {\em Progress in Probability}, pages 73--107.
\newblock Birkh\"auser, Basel, Switzerland, 2009.

\bibitem{ParkCity}
G.~F.~Lawler.
\newblock Schramm-Loewner Evolution (SLE).
\newblock In S.~Sheffield and T.~Spencer, editors,
\newblock {\em Statistical Mechanics},
\newblock volume 16 of {\em IAS/Park City Mathematics Series},  pages 231--296.
\newblock American Mathematical Society, Providence, RI, 2009.

\bibitem{contpaper}
G.~F.~Lawler.
\newblock Continuity of radial SLE and two-sided radial SLE at the terminal point.
\newblock Preprint,  2011. Available online at \texttt{arXiv:1104.1620}.


\bibitem{LawWerGF}
G.~F.~Lawler and B.~M.~Werness.
\newblock Multi-point Green's functions for SLE and an estimate of Beffara.
\newblock To appear, {\em Ann. Probab.} Available online at \texttt{arXiv:1011.3551}.

\bibitem{rohde_schramm}
S.~Rohde and O.~Schramm.
\newblock Basic properties of {SLE}.
\newblock {\em Ann. Math.}, 161:883--924, 2005.

\bibitem{Sch00}
O.~Schramm.
\newblock Scaling limits of loop-erased random walks and uniform spanning  trees.
\newblock {\em Israel J. Math.}, 118:221--288, 2000.

\end{thebibliography}
\end{document}